\documentclass[12pt]{amsart}
\usepackage{amssymb,latexsym,amsmath,amsthm,amscd}
\usepackage{setspace}
\usepackage[active]{srcltx}
\usepackage[all]{xy} \xyoption{arc}
\usepackage[left=3cm,top=2cm,right=3cm,bottom = 2cm]{geometry}
\usepackage{graphicx}
\usepackage[usenames,dvipsnames]{color}
\usepackage{tikz}
\usetikzlibrary{snakes}
\usetikzlibrary{arrows}

\usepackage{todonotes}
\usepackage{charter}
\usepackage{hyperref}
\hypersetup{
  colorlinks   = true, 
  urlcolor     = blue, 
  linkcolor    = blue, 
  citecolor   = blue 
}

\theoremstyle{plain}
\newtheorem{thm}{Theorem}
\newtheorem{lem}[thm]{Lemma}
\newtheorem{prop}[thm]{Proposition}
\newtheorem{cor}[thm]{Corollary}

\theoremstyle{definition}

\theoremstyle{remark}
\newtheorem{rmk}[thm]{Remark}

\DeclareMathOperator{\Tr}{Tr}

\DeclareMathOperator{\im}{im}

\DeclareMathOperator{\SL}{SL}
\DeclareMathOperator{\PSL}{PSL}

\newcommand*{\df}{\mathrel{\vcenter{\baselineskip0.5ex \lineskiplimit0pt
                     \hbox{\scriptsize.}\hbox{\scriptsize.}}} =}

\providecommand{\abs}[1]{\left\lvert#1\right\rvert}

\providecommand{\twomat}[4]{\left(\begin{matrix}#1&#2\\#3&#4\end{matrix}\right)}
\providecommand{\stwomat}[4]{\left(\begin{smallmatrix}#1&#2\\#3&#4\end{smallmatrix}\right)}

\providecommand{\lseries}[2]{#1(\!( #2 )\!)}

\newcommand{\QQ}{\mathbf{Q}}
\newcommand{\FF}{\mathbf{F}}

\newcommand{\Qbar}{\overline{\mathbf{Q}}}
\newcommand{\ZZ}{\mathbf{Z}}

\newcommand{\RR}{\mathbf{R}}

\DeclareMathOperator{\Stab}{Stab}

\begin{document}
\title[Noncongruence subgroups of index $7$]{The unbounded denominator conjecture for the noncongruence subgroups of index $7$}
\author{Andrew Fiori}
\address{University of Lethbridge}
\email{andrew.fiori@uleth.ca}
\author{Cameron Franc}
\address{McMaster University}
\email{franc@math.mcmaster.ca}

\thanks{The authors gratefully acknowledge financial support received from NSERC through their respective Discovery Grants, the financial support of the University of Lethbridge, and the use of computational resources made available through WestGrid and Compute Canada.}
\date{}

\begin{abstract}
We study modular forms for the minimal index noncongruence subgroups of the modular group. Our main theorem is a proof of the unbounded denominator conjecture for these groups, and we also provide a study of the Fourier coefficients of Eisenstein series for one of these minimal groups.
\end{abstract}
\maketitle

\setcounter{tocdepth}{1}
\tableofcontents

\section{Introduction}
\label{s:intro}
Finite index subgroups of the modular group $\Gamma = \PSL_2(\ZZ)$ play an important role in the study of algebraic curves thanks to a Theorem of Belyi \cite{Belyi}, which implies that curves of genus at least two defined over $\Qbar$ can be uniformized by such groups. The congruence subgroups correspond to the well-studied and fundamentally important modular curves, whereas the vast majority of curves correspond to finite index subgroups of $\Gamma$ that are \emph{not} defined by congruence conditions. Such noncongruence subgroups and their corresponding modular forms are much less well-understood than congruence groups and forms.

To date much of the work on noncongruence modular forms has focused on the following topics: Galois representations and congruences with congruence modular forms \cite{ASD}, \cite{Scholl3}, \cite{Scholl4}, \cite{Scholl5}, \cite{Scholl6}, \cite{Li1}, \cite{Li2}, \cite{HLV}, \cite{Long1}, \cite{Long2}, \cite{Long3}; the unbounded denominator conjecture \cite{ASD}, \cite{Scholl3}, \cite{Li3}, \cite{Long4}, \cite{Long5}, \cite{FrancMason1}, \cite{FrancMason2}, \cite{Gottesman}; moduli interpretations \cite{Chen}; spectral results \cite{PhilipsSarnak}, \cite{PhilipsSarnak2}, \cite{PhilipsSarnak3} ; algebraic properties of Eisenstein series \cite{Kubota}, \cite{Scholl1}, \cite{MurtyRamakrishnan}; computation of scattering matrices \cite{Kuhn}, \cite{Posingies}, \cite{BKP}. Given the vast generality inherent in the study of noncongruence modular forms, most papers in the subject have made progress by restricting to classes of groups that are more amenable to study than a general subgroup --- for example, one could consider the kernel of a character of a congruence subgroup, so that some finite power of each noncongruence form is a congruence modular form.

In the present paper our aim is to study some noncongruence subgroups that have not yet received particular focus and, from this perspective, it is natural to focus on subgroups of small index in $\Gamma$. The minimal index of a noncongruence subgroup of $\Gamma$ is known to be seven, and there are twenty-eight such subgroups of index seven that fall into four conjugacy classes. Our main theorem is a proof of the unbounded denominator conjecture for these groups:
\begin{thm}
  \label{t:main}
  Let $G$ be any of the noncongruence subgroups of $\Gamma$ of index seven, and let $f \in M_k(G) \setminus M_k(\Gamma)$ have algebraic Fourier coefficients at the cusp $\infty$. Then $f$ has unbounded denominators.
\end{thm}

Our proof of Theorem \ref{t:main} proceeds as follows:
\begin{enumerate}
\item[(a)] solve for a hauptmodul (or Belyi map);
\item[(b)] establish unbounded denominators at the prime $p=7$ for this hauptmodul\footnote{In two cases, called $U_1$ and $U_6$ below, we must use a prime over $7$ in a quadratic extension of $\QQ$.};
\item[(c)] use this result to prove Theorem \ref{t:main} in general.
\end{enumerate}
This argument can be adapted to many other groups, but it does not seem suited to generalization for at least two independent reasons: first, the diophantine problem involved in solving for a hauptmodul can be somewhat tricky in general and, second, it is not clear how the proof of unbounded denominators would generalize (in this paper we are aided by the fact that the index $[\Gamma \colon G] = 7$ is prime).

After we complete the proof of Theorem \ref{t:main} we turn to the study Eisenstein series. As in \cite{Scholl1}, \cite{MurtyRamakrishnan} we are able to establish the algebraicity of the Eisenstein series of weight two for more or less trivial reasons, but in higher weights we are only able to determine the complex phase of the Fourier coefficients. Knowledge of this phase follows from the study of the outer automorphism
\[
  \twomat{a}{b}{c}{d}\mapsto \twomat{a}{-b}{-c}{d}
\]
of $\Gamma$ acting on the subgroups of index seven. We end the paper with some numerical computations that indicate that there is more that one might be able to say about these Fourier coefficients in general, although we make no precise conjectures along these lines.

\subsection{Acknowledgments} The authors thank Jenna Rajchgot for some help in using \emph{Macaulay 2} to solve the $j$-equations in Sections \ref{s:details} and \ref{s:other}. We also thank \emph{WestGrid} and \emph{Compute Canada} for some computational support.

\subsection{Notation} Throughout the paper we use the following notation:
\begin{itemize}
\item[---] $\Gamma = \PSL_2(\ZZ)$, the modular group;
\item[---] $T = \stwomat 1101$, $S = \stwomat{0}{-1}{1}{0}$ and $R = ST$;
\item[---] if $G\subseteq \Gamma$ is of finite index, then $M(G) = \bigoplus_{k\geq 0} M_k(G)$ denotes the graded module of modular forms for $G$;
\item[---] $\zeta_n = e^{2\pi i/ n}$ and $q_n = e^{2\pi i\tau/n}$;
\item[---] $E_k \in M_k(\Gamma)$ denotes the Eisenstein series of weight $k$ with constant term normalized to equal $1$;
\item[---] $E_2$ denotes the normalized quasi-modular Eisenstein series of weight $2$;
\item[---] $j$ is the usual $j$-function, with constant term $744$.
\end{itemize}
\section{Noncongruence subgroups of index $7$}

\begin{figure}
\begin{tabular}{|c|c|}
  \hline
 \begin{tikzpicture}
\tikzset{decoration={snake,amplitude=.4mm,segment length=2mm,post length=0mm,pre length=0mm}}
\tikzset{vertex/.style = {shape=circle,draw,scale=.5}}
\tikzset{edge/.style = {->,> = latex'}}

  \node[vertex] (1) at (4,0) {};
  \node[vertex] (2) at (2.5,0.866) {};
  \node[vertex] (3) at (2.5,-0.866) {};
  \node[vertex] (4) at (-1,0) {};
  \node[vertex] (5) at (0.5,0.866) {};
  \node[vertex] (6) at (0.5,-0.866) {};
  \node[vertex] (7) at (-2,0) {};
  
  \draw[edge] (1) to [bend right=50] (2);
  \draw[edge] (2) to [bend right=50] (3);
  \draw[edge] (3) to [bend right=50] (1);   

  \draw[edge] (4) to [bend left=50] (5);
  \draw[edge] (5) to [bend left=50] (6);
  \draw[edge] (6) to [bend left=50] (4);
  
  \draw[decorate] (2) -- (5);
  \draw[decorate] (3) -- (6);
  \draw[decorate] (4) -- (7);
  \end{tikzpicture} $\phi_1$& \begin{tikzpicture}
\tikzset{decoration={snake,amplitude=.4mm,segment length=2mm,post length=0mm,pre length=0mm}}
\tikzset{vertex/.style = {shape=circle,draw,scale=.5}}
\tikzset{edge/.style = {->,> = latex'}}

\node[vertex] (1) at (4,0) {};
  \node[vertex] (2) at (2.5,0.866) {};
  \node[vertex] (3) at (2.5,-0.866) {};
  \node[vertex] (4) at (-1,0) {};
  \node[vertex] (5) at (0.5,0.866) {};
  \node[vertex] (6) at (0.5,-0.866) {};
  \node[vertex] (7) at (-2,0) {};
  
  \draw[edge] (1) to [bend left=50] (3);
  \draw[edge] (2) to [bend left=50] (1);
  \draw[edge] (3) to [bend left=50] (2);
  
  \draw[edge] (4) to [bend left=50] (5);
  \draw[edge] (5) to [bend left=50] (6);
  \draw[edge] (6) to [bend left=50] (4);
  
  \draw[decorate] (2) -- (5);
  \draw[decorate] (3) -- (6);
  \draw[decorate] (4) -- (7);
\end{tikzpicture}$\phi_2$\\
  \hline
 \begin{tikzpicture}
\tikzset{decoration={snake,amplitude=.4mm,segment length=2mm,post length=0mm,pre length=0mm}}
\tikzset{vertex/.style = {shape=circle,draw,scale=.5}}
\tikzset{edge/.style = {->,> = latex'}}

  \node[vertex] (1) at (4,0) {};
  \node[vertex] (2) at (2.5,0.866) {};
  \node[vertex] (3) at (2.5,-0.866) {};
  \node[vertex] (4) at (-1,0) {};
  \node[vertex] (5) at (0.5,0.866) {};
  \node[vertex] (6) at (0.5,-0.866) {};
  \node[vertex] (7) at (-2,0) {};
  
  \draw[edge] (1) to [bend right=50] (2);
  \draw[edge] (2) to [bend right=50] (3);
  \draw[edge] (3) to [bend right=50] (1);   

  \draw[edge] (4) to [bend left=50] (5);
  \draw[edge] (5) to [bend left=50] (6);
  \draw[edge] (6) to [bend left=50] (4);
  
  \draw[decorate] (2) -- (5);
  \draw[decorate] (3) -- (1);
  \draw[decorate] (4) -- (7);
  \end{tikzpicture} $\phi_3$& \begin{tikzpicture}
\tikzset{decoration={snake,amplitude=.4mm,segment length=2mm,post length=0mm,pre length=0mm}}
\tikzset{vertex/.style = {shape=circle,draw,scale=.5}}
\tikzset{edge/.style = {->,> = latex'}}

\node[vertex] (1) at (4,0) {};
  \node[vertex] (2) at (2.5,0.866) {};
  \node[vertex] (3) at (2.5,-0.866) {};
  \node[vertex] (4) at (-1,0) {};
  \node[vertex] (5) at (0.5,0.866) {};
  \node[vertex] (6) at (0.5,-0.866) {};
  \node[vertex] (7) at (-2,0) {};
  
  \draw[edge] (3) to [bend right=50] (1);
  \draw[edge] (1) to [bend right=50] (2);
  \draw[edge] (2) to [bend right=50] (3);
  
  \draw[edge] (5) to [bend right=50] (4);
  \draw[edge] (6) to [bend right=50] (5);
  \draw[edge] (4) to [bend right=50] (6);
  
  \draw[decorate] (2) -- (5);
  \draw[decorate] (3) -- (1);
  \draw[decorate] (4) -- (7);
\end{tikzpicture} $\phi_4$\\
  \hline
   \begin{tikzpicture}
\tikzset{decoration={snake,amplitude=.4mm,segment length=2mm,post length=0mm,pre length=0mm}}
\tikzset{vertex/.style = {shape=circle,draw,scale=.5}}
\tikzset{edge/.style = {->,> = latex'}}

  \node[vertex] (1) at (4,0) {};
  \node[vertex] (2) at (2.5,0.866) {};
  \node[vertex] (3) at (2.5,-0.866) {};
  \node[vertex] (4) at (-1,0) {};
  \node[vertex] (5) at (0.5,0.866) {};
  \node[vertex] (6) at (0.5,-0.866) {};
  \node[vertex] (7) at (-2,0) {};
  
  \draw[edge] (1) to [bend right=50] (2);
  \draw[edge] (2) to [bend right=50] (3);
  \draw[edge] (3) to [bend right=50] (1);   

  \draw[edge] (4) to [bend left=50] (5);
  \draw[edge] (5) to [bend left=50] (6);
  \draw[edge] (6) to [bend left=50] (4);
  
  \draw[decorate] (2) -- (5);
  \draw[decorate] (4) -- (7);
  \end{tikzpicture} & \begin{tikzpicture}
\tikzset{decoration={snake,amplitude=.4mm,segment length=2mm,post length=0mm,pre length=0mm}}
\tikzset{vertex/.style = {shape=circle,draw,scale=.5}}
\tikzset{edge/.style = {->,> = latex'}}

  \node[vertex] (1) at (4,0) {};
  \node[vertex] (2) at (2.5,0.866) {};
  \node[vertex] (3) at (2.5,-0.866) {};
  \node[vertex] (4) at (-1,0) {};
  \node[vertex] (5) at (0.5,0.866) {};
  \node[vertex] (6) at (0.5,-0.866) {};
  \node[vertex] (7) at (-2,0) {};
  
  \draw[edge] (1) to [bend right=50] (2);
  \draw[edge] (2) to [bend right=50] (3);
  \draw[edge] (3) to [bend right=50] (1);   

  \draw[edge] (4) to [bend left=50] (5);
  \draw[edge] (5) to [bend left=50] (6);
  \draw[edge] (6) to [bend left=50] (4);
  
  \draw[decorate] (3) -- (6);
  \draw[decorate] (4) -- (7);
  \end{tikzpicture}\\
  \hline
\end{tabular}
\caption{Cycle types for the homomorphisms $\Gamma \to S_7$ with transitive image. The squiggly lines correspond to the image of $S$ and the arrows correspond to the image of $R$. The graphs are unlabeled since we are interested in homomorphisms up to conjugation in $S_7$. The last row corresponds to congruence subgroups.}
\label{f:s7}
\end{figure}
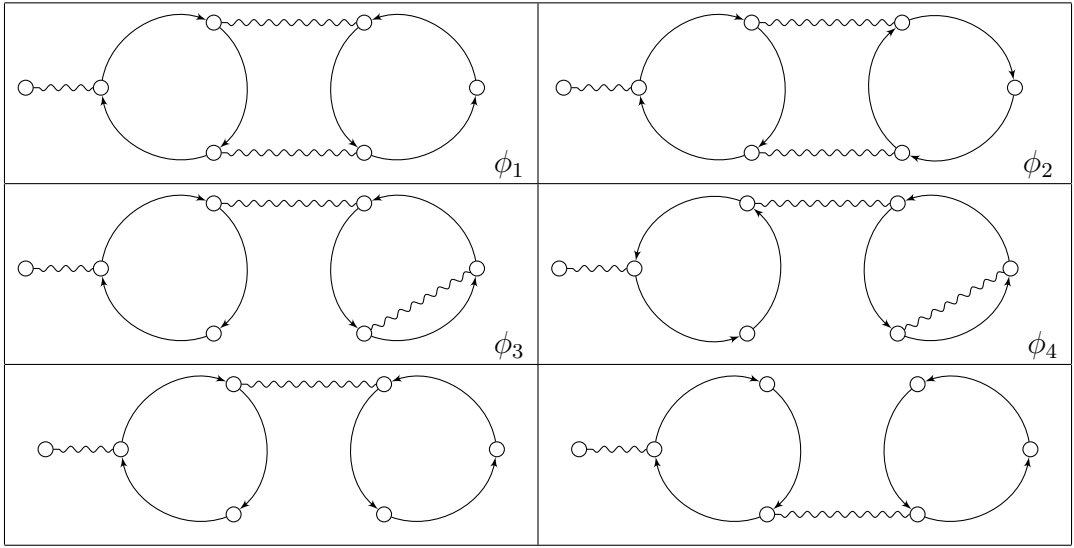

If $G\subseteq \Gamma$ is of index $7$, then the action of $G$ on the cosets defines a homomorphism $\Gamma \to S_7$ with transitive image. Since $\Gamma = \SL_2(\ZZ)/\{\pm 1\}$ is freely generated by order $2$ and $3$ elements $S$ and $R$, respectively, there are six such homomorphisms up to conjugation. They are depicted in Figure \ref{f:s7}, where the squiggly lines describe the image of $S$, and the arcs describe the image of $R$. The last row in Figure \ref{f:s7} corresponds to congruence subgroups of level $7$, and so we shall only focus on the first two rows. Call the corresponding (conjugacy classes) of homomorphisms $\phi_j$ as in Figure \ref{f:s7}. To be more precise, we pick the representatives for the conjugacy classes as in Table \ref{t:homs}.
\begin{table}
  \renewcommand{\arraystretch}{1.5}
  \begin{tabular}{|c|c|c|c|c|}
    \hline
    &$\phi_1$ & $\phi_2$ & $\phi_3$ & $\phi_4$\\
    \hline
    $S \mapsto$& $(12)(34)(56)$ & $(12)(34)(56)$&$(12)(34)(67)$& $(12)(34)(67)$\\
    \hline
    $R \mapsto$& $(235)(467)$&$(235)(764)$&$(235)(467)$&$(253)(467)$\\
    \hline
    $T \mapsto$& $(1245)(367)$&$(12475)(36)$&$(124735)$& $(125473)$\\
    \hline
    $\abs{\im \phi_j}$ & $7!$ & $7!$ & $42$ & $42$\\
    \hline
  \end{tabular}
  \caption{Data for the noncongruence homomorphisms $\Gamma \to S_7$ with transitive image.}
  \label{t:homs}
\end{table}

Note that $\phi_1$ and $\phi_2$ are surjective onto $S_7$, while $\phi_3$ and $\phi_4$ have images of order $42$. Let
\begin{align*}
  G_j &= \phi_1^{-1}(\Stab_{\im \phi_1}(j)), &H_j &= \phi_2^{-1}(\Stab_{\im \phi_2}(j)),\\
  U_j &= \phi_3^{-1}(\Stab_{\im \phi_3}(j)), &V_j &= \phi_4^{-1}(\Stab_{\im \phi_4}(j)).
\end{align*}
The following result is known to experts, but we did not find a suitable reference in the literature.
\begin{thm}
  The groups $G_j$, $H_j$, $U_j$ and $V_j$ for $j=1,\ldots, 7$ are the noncongruence subgroups of $\Gamma$ of smallest index.
\end{thm}
\begin{proof}
  Theorem 5 of \cite{Wohlfahrt} shows that all subgroups of $\Gamma$ of index $\leq 6$ are congruence. That these groups are noncongruence will follow by our proof of unbounded denominators, but this can also be proved by more elementary means: for example, for the $G_j$ and $H_j$ it follows easily by the simplicity of $A_7$. For the $U_j$, note that $T$ has order $6$ in the quotient $\Gamma/\ker \phi_3$. Hence by Theorem 2 in $\S 3.1$ of \cite{Schoeneberg}, if $U_j$ were congruence it would have to contain $\Gamma(6)$. But by elementary group theory one sees that there are no congruence subgroups of index $7$ and level $6$. Therefore the $U_j$ (and similarly the $V_j$) are noncongruence. By considering the image of $T$, it is clear that all of these subgroups are distinct except possibly for some identity $U_i=V_j$, and without loss of generality we may assume $U_1 = V_j$ for some $j$. Since $R \in U_1$, and the only $V_j$ that contains $R$ is $V_1$, we thus would have $U_1 = V_1$. But one can easily check that $U_1$ contains $SRSR^2S$, while $V_1$ does not. Therefore $U_1 \neq V_1$ and hence $U_i\neq V_j$ for all $i$ and $j$.

  The only other subgroups of index $7$ come from the last line in Figure \ref{f:s7}, but those homomorphisms yield the congruence subgroups of level $7$ and index $7$.
\end{proof}
\begin{rmk}
In general, the number of subgroups of index $n$ in $\Gamma$ can be counted using Exercise 5.13 of \cite{Stanley}.
\end{rmk}

In the remainder of this section we shall describe some group theoretic data for $G_1$, $H_1$, $U_1$ and $V_1$ that will be useful for what follows. The analogous data for the other groups can be obtained by conjugation.

\subsection{Data for $G_1$}
A fundamental domain for $G_1$ is given in Figure \ref{f:fundom1} on page \pageref{f:fundom1}.
\begin{figure}
  \includegraphics[scale=0.7]{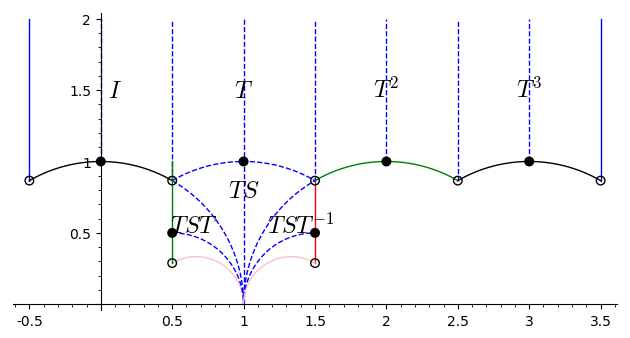}
  \caption{A fundamental domain for $G_1$. The colours of the edges describe the edge pairing.}
  \label{f:fundom1}
\end{figure}
 The elliptic points are represented by $\zeta_3$ and $\tfrac{1}{2}(3+i)$. The edge pairing is defined by the following matrices:
\begin{align*}
  T^4 &= \twomat 1401, & C_2 &= \twomat{-3}{1}{-1}{0},\\
  C_1 &= \twomat{4}{-3}{3}{-2}, & E_1 &= \twomat{3}{-5}{2}{-3}.
\end{align*}
Here $T^4$ identifies the vertical blue sides, $C_1$ identifies the pink edges, $C_2$ identifies the green and black edges with their pairs, and $E_1$ identifies the red edge with itself. Note that $R = C_2^{-1}C_1$, so that if we use $R$ as a generator, we can dispense with $C_2$. Similarly, $C_1 = E_1^{-1}CR^{-1}$, so that we can also dispense with $C_1$. By standard results on Fuchsian groups, one obtains the following presentation for $G_1$:
\[G_1 = \langle T^4,E_1,R \mid E_1^2 = R^3=1\rangle.\]

\subsection{Data for $H_1$} A fundamental domain for $H_1$ is given in Figure \ref{f:fundom2} on page \pageref{f:fundom2}.
\begin{figure}
  \includegraphics[scale=0.7]{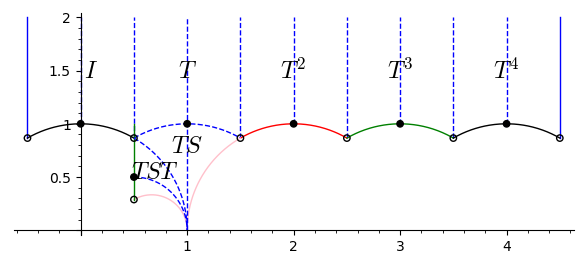}
  \caption{A fundamental domain for $H_1$. The colours of the edges describe the edge pairing.}
  \label{f:fundom2}
\end{figure}
 The elliptic points are represented by $\zeta_3$ and $2+i$. The edge pairing is defined by the following matrices:
\begin{align*}
  T^5 &= \twomat 1501, & D_2 &= \twomat{-4}{1}{-1}{0},\\
  D_1 &= \twomat{1}{-2}{2}{-3}, & E_2 &= \twomat{2}{-5}{1}{-2}.
\end{align*}
Here $T^5$ identifies the vertical blue sides, $D_1$ identifies the pink edges, $D_2$ identifies the green and black edges with their pairs, and $E_2$ identifies the red edge with itself. Similarly to above we obtain a presentation
\[H_1 = \langle T^5,E_2,R \mid E_2^2 = R^3=1\rangle.\]
In particular $G_1 \cong H_1$ but this will play no role in what follows.

\subsection{Data for $U_1$} A fundamental domain for $U_1$ is given in Figure \ref{f:fundom3} on page \pageref{f:fundom3}.
\begin{figure}
  \includegraphics[scale=0.8]{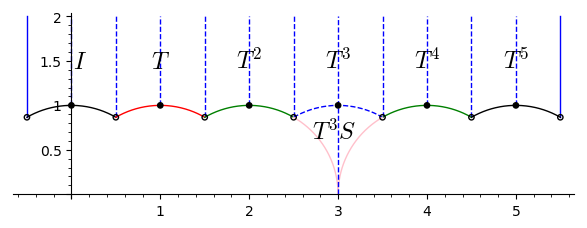}
  \caption{A fundamental domain for $U_1$. The colours of the edges describe the edge pairing.}
  \label{f:fundom3}
\end{figure}
 The elliptic points are represented by $\zeta_3$ and $1+i$. The edge pairing is defined by the following matrices:
\begin{align*}
  T^6 &= \twomat 1601, & A_2 &= \twomat{-5}{1}{-1}{0},\\
  A_1 &= \twomat{-2}{9}{-1}{4}, & E_3 &= \twomat{1}{-2}{1}{-1}.
\end{align*}
Here $T^6$ identifies the vertical blue sides, $A_1$ identifies the pink and green edges with their pairs, $A_2$ identifies the black edges, and $E_3$ identifies the red edge with itself. Similarly to above we obtain a presentation
\[U_1 = \langle T^6,E_3,R \mid E_3^2 = R^3=1\rangle.\]

\subsection{Data for $V_1$} A fundamental domain for $V_1$ is given in Figure \ref{f:fundom4} on page \pageref{f:fundom4}.
\begin{figure}
  \includegraphics[scale=0.8]{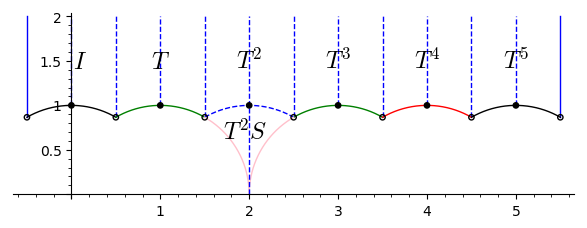}
  \caption{A fundamental domain for $V_1$. The colours of the edges describe the edge pairing.}
  \label{f:fundom4}
\end{figure}
 The elliptic points are represented by $\zeta_3$ and $4+i$. The edge pairing is defined by the following matrices:
\begin{align*}
  T^6 &= \twomat 1601, & B_2 &= \twomat{-5}{1}{-1}{0},\\
  B_1 &= \twomat{-1}{4}{-1}{3}, & E_4 &= \twomat{4}{-17}{1}{-4}.
\end{align*}
Here $T^6$ identifies the vertical blue sides, $B_1$ identifies the pink and green edges with their pairs, $B_2$ identifies the black edges, and $E_4$ identifies the red edge with itself. Similarly to above we obtain a presentation
\[V_1 = \langle T^6,E_4,R \mid E_4^2 = R^3=1\rangle.\]

\section{Outline of proof of Theorem \ref{t:main}}
We begin the proof of Theorem \ref{t:main} with an elementary reduction.
\begin{lem}
  \label{l:reduction1}
  To prove Theorem \ref{t:main} it suffices to treat the cases $G = G_1$, $G_3$, $H_1$, $H_3$, $U_1$, $U_6$, $V_1$ and $V_6$.
\end{lem}
\begin{proof}
We first consider the case of $G_1$, which is conjugate with $G_2$, $G_4$ and $G_5$ via powers of $T$. Therefore, if $f \in M_k(G_1)$ has Fourier expansion $f = \sum_{n\geq 0}a_nq_4^n$ where $q_4= e^{2\pi i\tau/4}$, then the forms $f(\tau+m)=\sum_{n\geq 0} a_ni^{mn}q_4^n$ are forms on the various conjugate groups, and vice versa. Hence Theorem \ref{t:main} holds for $G_1$ if and only if it holds for any one of $G_1$, $G_2$, $G_4$ or $G_5$. An identical arguments applies to the other cases.
\end{proof}
While not strictly necessary, we can describe the structure of $M(G)$ as an $M(\Gamma)$-module, which then immediately gives the dimensions of the graded pieces of $M(G)$ (see Table \ref{t:dims}).
\begin{lem}
  \label{l:reduction2}
  Let $G$ be any of the noncongruence subgroups of $\Gamma$ of index $7$. Then there exists a free-basis for $M(G)$ as an $M(\Gamma)$-module with generators in weights $0$, $2$, $4$, $4$, $6$, $6$ and $8$ such that the generators have algebraic Fourier coefficients.
\end{lem}
\begin{proof}
  Let $\rho$ be the representation of $\Gamma$ obtained from the permutation representation of $G$ on its cosets in $\Gamma$, so that if $M(\rho)$ is the corresponding space of vector-valued modular forms, then there is an isomorphism $M(G) \cong M(\rho)$. Since $\rho$ decomposes as the trivial representation plus an even $6$-dimensional irreducible representation for each of these $G$, which at the level of scalar forms corresponds to the decomposition $M(G) = M(\Gamma) \oplus \ker \Tr$ where $\Tr \colon M(G) \to M(\Gamma)$ is the congruence trace, one can use Riemann-Roch to find the weights of a free-basis for $M(\rho)$ of the desired form --- see Example 7.4 of \cite{CandeloriFranc} where $\phi_1$ and $\phi_2$ are treated explicitly. Both $\phi_3$ and $\phi_4$ are analogous to that case, as the local monodromies around elliptic points are conjugate in all cases, and the sum of the exponents of $\rho(T)$ is $5/2$ in each of these four cases, as can be easily read off from the cycle type of $\psi_j(T)$ (not that this does not hold for the congruence subgroups of index $7$ where $T$ acts as a seven-cycle).

  To obtain a free-basis with algebraic Fourier coefficients, one could for example diagonalize $\rho(T)$ and use the Frobenius method (in fact, in the course of the proof of Theorem \ref{t:main} we will write down an explicit free-basis with algebraic Fourier coefficients). Since $\rho(T)$ can diagonalized over a finite extension of $\QQ$, algebraicity of Fourier coefficients will be preserved under this operation.
\end{proof}

\begin{table}
  \begin{tabular}{|c|c|}
    \hline
    $k$ & $\dim M_k(G)$\\
    \hline
    $0$&$1$\\
    $2$&$1$\\
    $4$&$3$\\
    $6$&$4$\\
    $8$&$5$\\
    $10$&$6$\\
    $k\geq 12$&$\dim M_{k-12}(G)+7$\\
    \hline
\end{tabular}
\caption{Dimension of $M_k(G)$ for the noncongruence subgroups $G$ of index $7$.}
\label{t:dims}
\end{table}

\begin{rmk}
One can easily use the hauptmoduls described below to compute explicit free-bases for $M(G)$ as in Lemma \ref{l:reduction2}, through taking derivatives and products of forms, but we have no need for such a free-basis in this paper.
\end{rmk}

In light of Lemmas \ref{l:reduction1} and \ref{l:reduction2}, we can now proceed as follows:
\begin{enumerate}
\item describe a hauptmodul for each of the genus $0$ groups $G_1$, $G_3$, $H_1$, $H_3$, $U_1$, $U_6$, $V_1$ and $V_6$;
\item establish unbounded denominators for each hauptmodul;
\item express modular forms in $M(G)$ as forms of level one times rational functions in the hauptmodul, and deduce unbouded denominators as a result.
\end{enumerate}
We shall give all of the details for the group $G_1$ in Section \ref{s:details}, and then in Section \ref{s:other} we shall summarize the key facts that allow one to carry out the same argument for the other groups.

\section{Details for the group $G_1$}
\label{s:details}
Following Atkin--Swinnerton-Dyer \cite{ASD}, we compute a hauptmodul for the genus $0$ group $G_1$. Note that this hauptmodul is an example of a uniformizing Belyi map, and there exists an extensive literature on computing Belyi maps -- see \cite{Voight1}, \cite{Voight2} and the references contained therein for more information.

Recall that $j = \tfrac 1q + 744 + 196884q +\cdots$. In the setting of $G_1$, ASD solve for a hauptmodul $z = q_4^{-1} + 0 + O(q_4)$ (so that $q_4 = \xi$ in the ASD notation) by introducing polynomials:
\begin{align*}
  a_1 &= z+c_1,\\
  f_3 &= z^2+c_2z+c_3,\\
  e_3 &= z+c_4,\\
  f_2 &= z^3+c_5z^2+c_6z+c_7,\\
  e_2 &= z+c_8,
\end{align*}
and then solving for the unknown $c_j$'s via the $j$-equations:
\begin{align*}
  ja_1^3 &= f_3^3e_3,\\
  (j-1728)a_1^3 &= f_2^2e_2.
\end{align*}
Eliminating $j$ from these two equations gives the system of nonlinear equations:
\begin{align*}
  0 &= -c_{3}^{3} c_{4} + 1728 c_{1}^{3} + c_{7}^{2} c_{8},\\
  0 &=-3 c_{2} c_{3}^{2} c_{4} - c_{3}^{3} + 2 c_{6} c_{7} c_{8} + 5184 c_{1}^{2} + c_{7}^{2},\\
  0 &=-3 c_{2}^{2} c_{3} c_{4} - 3 c_{2} c_{3}^{2} - 3 c_{3}^{2} c_{4} + c_{6}^{2} c_{8} + 2 c_{5} c_{7} c_{8} + 2 c_{6} c_{7} + 5184 c_{1},\\
  0 &=-c_{2}^{3} c_{4} - 3 c_{2}^{2} c_{3} - 6 c_{2} c_{3} c_{4} + 2 c_{5} c_{6} c_{8} - 3 c_{3}^{2} + c_{6}^{2} + 2 c_{5} c_{7} + 2 c_{7} c_{8} + 1728,\\
  0 &=-c_{2}^{3} - 3 c_{2}^{2} c_{4} + c_{5}^{2} c_{8} - 6 c_{2} c_{3} - 3 c_{3} c_{4} + 2 c_{5} c_{6} + 2 c_{6} c_{8} + 2 c_{7},\\
  0 &=-3 c_{2}^{2} - 3 c_{2} c_{4} + c_{5}^{2} + 2 c_{5} c_{8} - 3 c_{3} + 2 c_{6},\\
  0 &=-3 c_{2} - c_{4} + 2 c_{5} + c_{8}.
\end{align*}
Furthermore, if we expand the $j$-equations in $z$ and compare constant terms, we get an additional linear equation $3c_1-3c_2-c_4 = 0$. Finally, we must insist that $a_1$, $f_3$, $e_3$, $f_2$ and $e_2$ have distinct roots as polynomials in $z$. This means we have the nonequalities:
\begin{align*}
  (c_1-c_4)(c_1-c_8)(c_4-c_8) &\neq 0,\\
  c_j^2+c_2c_j+c_3 &\neq 0 \quad\quad (j=1,4,8),\\
  c_j^3+c_5c_j^2+c_6c_j+c_7 &\neq 0 \quad\quad (j = 1,4,8),\\
  c_2^2 - 4c_3 &\neq 0,\\
  c_5^2c_6^2-4c_6^3-4c_5^3c_7-27c_7^2+18c_5c_6c_7 &\neq 0.
\end{align*}
We put these into \emph{Macaulay 2} and performed a Grobner basis computation to find the following substitutions:
\begin{align*}
  c_1 &= \tfrac{2}{3}c_5+\tfrac{1}{3} c_8,\\
  c_3 &= 2c_2^2-2c_2c_5+\tfrac{1}{3}c_5^2-c_2c_8+\tfrac{2}{3}c_5c_8+\tfrac{2}{3}c_6,\\
  c_4 &= -3c_2+2c_5+c_8,\\
  c_7 &= -\tfrac{3}{7}(c_2^3+\tfrac{3}{2}c_2c_5^2-\tfrac{119}{81}c_5^3+\tfrac 23 c_2^2c_8+c_2c_5c_8-\tfrac{70}{27}c_5^2c_8+\tfrac{7}{6}c_2c_8^2- \tfrac{52}{81}c_8^3\\
  &\quad\quad- \tfrac{11}{3}c_2c_6 + 3c_5c_6 + \tfrac{32}{9}c_6c_8).
\end{align*}
This leaves a system of $5$ equations in the unknowns $c_2$, $c_5$, $c_6$ and $c_8$. With some effort involving saturating with respect to the last nonequality condition, we managed to find the following solution: if $u = \sqrt[4]{-7}/7^2$ then
\begin{align*}
  c_1 &= 168u, &  c_2 &= 256u,\\
  c_3 &= 10869u^2,&  c_4 &= -264u,\\
  c_5 &= 160u,&  c_6 &= -28968u^2,\\
  c_7 &= -5900544u^3,&  c_8 &= 184u.
\end{align*}
Given this, one can recursively solve for the $q_4$-expansion coefficients of $z$ using the $j$-equations. In Table \ref{t:hauptmodul} on page \pageref{t:hauptmodul} we list the rational part of the Fourier coefficients of the hauptmodul $z$.

\begin{table}
  \begin{tabular}{|c|c|}
    \hline
    $n$ & $a_n/u^{n+1}$\\
    \hline
    $-1$ & $ 1 $ \\
$0$ & $ 0 $ \\
$1$ & $ 2^{2} \cdot 3 \cdot 7 \cdot 173 $ \\
$2$ & $ 2^{11} \cdot 7 \cdot 43 $ \\
$3$ & $ -1 \cdot 2 \cdot 3 \cdot 7 \cdot 173 \cdot 199 $ \\
$4$ & $ -1 \cdot 2^{14} \cdot 3^{9} \cdot 7 $ \\
$5$ & $ -1 \cdot 2^{3} \cdot 7 \cdot 17 \cdot 89 \cdot 1969543 $ \\
$6$ & $ -1 \cdot 2^{13} \cdot 3 \cdot  \cdot 5273 \cdot 47339 $ \\
$7$ & $ 3^{2} \cdot 7 \cdot 11 \cdot 19 \cdot 26353729 $ \\
$8$ & $ 2^{17} \cdot 3^{10} \cdot 7 \cdot 31 \cdot 67 \cdot 131 $ \\
$9$ & $ 2^{2} \cdot 3 \cdot 7 \cdot 11869625271733553 $ \\
$10$ & $ 2^{12} \cdot 3 \cdot 7 \cdot 17 \cdot 1579 \cdot 36677 \cdot 385321 $ \\
$11$ & $ 2 \cdot 3 \cdot 7^2 \cdot 7204271 \cdot 2154711443 $ \\
    \hline
  \end{tabular}
  \caption{Normalized Fourier coefficients of the hauptmodul for $G_1$.}
  \label{t:hauptmodul}
\end{table}

\begin{lem}
The hauptmodul has unbounded denominators.
\end{lem}

\begin{proof}
We note the minimal polynomial for the hauptmodul, $z$, as it defines a finite extension of $\Qbar(j)$ is
\[ (z^2+c_2z + c_3)^3(z+c_4) - j(z+c_1)^3. \]
We now perform the changes of variables
\[   z = u \hat{z} \]
 where $u = \sqrt[4]{-7}/7^2$, as above. We see that $\hat{z}$ satisfies
 \begin{align*} 
   \hat{z}^7 + 504\hat{z}^6 +& 26544\hat{z}^5 - 27020672\hat{z}^4 - 6349147392\hat{z}^3\\&
       - 568400910336\hat{z}^2 - 22777684586496\hat{z} - 341511404027904\\
   &  - 7^7j(\hat{z}^3 + 504\hat{z}^2 + 84672\hat{z} + 4741632).  \end{align*}
 To better understand the denominators in the Laurent expansion we shall formally substitute
 \[  q_4 =  u^{-1} \hat{q} \]
 and study $\hat{z}$ as a Laurent series in $\hat{q}$.
  Additionally, we renormalize $j = -\hat{j}/7^7$ so that $\hat{j} = \frac{1}{\hat{q}^4} \pmod{7}$.
  We then have
 \begin{align*}   \hat{z}^7 + 504\hat{z}^6 +& 26544\hat{z}^5 - 27020672\hat{z}^4 - 6349147392\hat{z}^3\\&
                - 568400910336\hat{z}^2 - 22777684586496\hat{z} - 341511404027904\\&
                 + \hat{j}(\hat{z}^3 + 504\hat{z}^2 + 84672\hat{z} + 4741632). 
                \end{align*}
  By an application of Hensel's lemma we can conclude that the coefficients of $\hat{z}$, as a Laurent series in $\hat{q}$, are integers.
 
 We may thus reduce the series $\hat{z}$ modulo $7$ and notice that the result satisfies the minimal polynomial 
 \[ x^7 + \hat{j}x^3 + 2 \pmod{7}\]
 over the function field $\FF_7(\hat{j})$. Solutions to this equation in Laurent series $\lseries{\FF_7}{\hat{q}}$ must have infinitely many non-zero coefficients. Indeed, a non-constant polynomial cannot satisfy a polynomial equation of degree greater than $0$.
 
 This implies that $\hat{z}$, and hence $z$, has unbounded denominators as a Laurent series when expressed in the variable $q_4$.
\end{proof}

\begin{cor}
  \label{c:ubd}
Any element of $\Qbar(z)$ not in $\Qbar(j)$ has unbounded denominators.
\end{cor}
\begin{proof}
If $f\in \Qbar(z)\setminus \Qbar(j)$ then the field extension it generates satisfies
\[ \Qbar(f)/ \Qbar(j) =  \Qbar(z)/ \Qbar(j) \]
and hence it follows that $z$ can be expressed as $P(f)$ where $P$ is a polynomial in $\Qbar(j)$.
By clearing denominators from the coefficients of $P$ we can write
\[ z R_1(j) = P_2(f) \]
where now $R_1\in \Qbar[j]$ and $P_2$ has coefficients in $ \Qbar[j]$.
As the left hand side, $z R_1(j)$, has unbounded denominators, so too must $f$.
\end{proof}

\begin{proof}[Completion of proof of Theorem \ref{t:main} for $G_1$]
 Every form  $f \in M_k(G)\setminus M_k(\Gamma)$ of weight at least $4$ with algebraic Fourier coefficients can be expressed as $f =E_k P(z)$ for $P(z) \in \Qbar(z)\setminus \Qbar(j)$. By Corollary \ref{c:ubd} $P(z)$ has unbounded denominators, and hence so does $f$. If $f \in M_2(G) \setminus M_2(\Gamma)$ then instead write $f = (E_6/E_4)P(z)$ for $P(z) \in \Qbar(z) \setminus \Qbar(j)$ and then the same argument applies.
\end{proof}

\section{Some details for the other groups}
\label{s:other}

Since the data about elliptic points for all the index $7$ subgroups agree, the degrees of $a_1$, $f_3$, $e_3$, $f_2$ and $e_2$ are the same as for the group $G_1$. Therefore in each case below we retain our notation from Section \ref{s:details} for these polynomials in terms of unknowns $c_1$ through $c_8$. For all but $H_1$ we were able to solve the equations through a mixture of saturating with respect to the ASD nonequalities for the $j$-equations, as well as using Grobner bases. Unfortunately $H_1$ does not meet the locus defined by the nonequalities and this strategy did not help there. For $H_1$ we instead performed a sequence of projections all the way down to the variable $c_8$ and then we were able to decompose the ideal. This yielded two additional spurious components, as well as the unique correct solution below.

Let $\psi \colon \Gamma \to \Gamma$ be the outer automorphism given by conjugation with $\stwomat{-1}001$, so that $\psi\stwomat abcd = \stwomat a{-b}{-c}d$. Notice that $\psi(T) = T^{-1}$ and $\psi(S) = S^{-1}$, and that $\psi$ maps congruence subgroups to congruence subgroups. Since the $j$-equations do not distinguish between complex conjugate subgroups, we shall describe how complex conjugation acts on our groups.
\begin{lem}
  \label{l:outer}
  The outer automorphism $\psi$ permutes the $G_j$ among themselves, and likewise for the $H_j$. In both cases the action is given by the permutation $(12)(36)(45)$. On the other hand, $\psi$ maps the $U_j$ groups to the $V_j$ groups as follows:
  \begin{align*}
    \psi(U_1) &= V_2, &\psi(U_2) &= V_1, &\psi(U_3) &= V_4, &\psi(U_4) &= V_3, \\
    \psi(U_5) &= V_5, &\psi(U_6) &= V_6, &\psi(U_7) &= V_7.
  \end{align*}
\end{lem}
\begin{proof}
This can be proved using the presentations for each of these groups. Conjugate the generators of each subgroup and test what group contains the result. This is a finite computation that is easily performed on a computer.
\end{proof}
What this lemma means in practical terms is that we only need to solve the $j$-equations for $U_1$ and $U_6$, as the Galois orbits of these solutions will contain the hauptmoduls for all of the $U_j$ and $V_j$.

The details in establishing unbounded denominators are the same as for $G_1$, save that certain constants change, and so we shall omit them.

\subsection{The conjugates of $G_3$}
In this case the $j$-equations read
\begin{align*}
  ja_1^4 &= f_3^3e_3,\\
  (j-1728)a_1^4 &= f_2^2e_2.
\end{align*}
If $u = \sqrt[3]{-2/7}/7^2$, then the solution to the $j$-equations is:
\begin{align*}
  c_1 &= -462u, &c_2 &= -444u,\\
  c_3 &= -148284u^2, &c_4 &=-516u,\\
  c_5 &=-1422u, & c_6 &=822204u^2,\\
  c_7 &= -185029704 u^3, & c_8 &=996u.
\end{align*}
The first few terms of the Fourier expansion of the hauptmodul are given in Tables \ref{tab:hauptG3}.

\begin{table}[t]
  \begin{tabular}{|c|c|}
    \hline
    $n$ & $a_n/u^{n+1}$\\
    \hline
    -1 &$1 $\\
    0 & $0$ \\
1&$148932$\\
2&$-71333864/2$\\
3&$14784602112/2$\\
4&$-2720037481056/2$\\
5&$926140535244764/2^2$\\
6&$ -147594381291749376/2^2$\\
7&$ 22341564325891713168/2^2$\\
8&$ -6482694981105850075968/2^3$\\
9&$ 907550467150406926565376/2^3$\\
10&$ -123344662799290912907945472/2^3$\\
11&$ 32655462531659638680360877638/2^4$\\
\hline
\end{tabular}
\caption{Normalized Fourier coefficients of the hauptmodul for $G_3$}\label{tab:hauptG3}
\end{table}

\begin{rmk}
Note that the apparent powers of $2$ in the denominators are not necessary in most of the terms displayed.
Though if the table were extended with the same pattern there would be infinitely many terms where the numerator is odd.
However, these powers of $2$ cancel with those from the power of $u$ and are not unbounded denominators in the actual $q$-expansion.
\end{rmk}

\subsection{The conjugates of  $H_1$}
In this case the $j$-equations read
\begin{align*}
  ja_1^2 &= f_3^3e_3,\\
  (j-1728)a_1^2 &= f_2^2e_2.
\end{align*}
If $u = \sqrt[5]{-7^3}/7^2$, then the solution to the $j$-equations is:
\begin{align*}
  c_1 &= 28u, &c_2 &= 51u,\\
  c_3 &= -636u^2, &c_4 &=-97u,\\
  c_5 &=-18u, & c_6 &= -2979u^2,\\
  c_7 &= -111348u^3, & c_8 &=92u.
\end{align*}
The first few terms of the Fourier expansion of the hauptmodul are given in Table \ref{tab:hauptH1}.
\begin{table}[b]
  \begin{tabular}{|c|c|}
    \hline
    $n$ & $a_n/u^{n+1}$\\
    \hline
    -1 & $1 $\\
    0 & $0$ \\
1&$    1946$\\
2&$17780$\\
3&$813295$\\
4&$-20472508$\\
5&$-194969600$\\
6&$-21590535732$\\
7&$-86533770365$\\
8&$-5540827925500$\\
9&$121544077700080$\\
10&$954435095756800$\\
11&$97227702559110739$\\
    \hline
\end{tabular}
\caption{Normalized Fourier coefficients of the hauptmodul for $H_1$}\label{tab:hauptH1}
\end{table}

\subsection{The conjugates of $H_3$}
In this case the $j$-equations read
\begin{align*}
  ja_1^5 &= f_3^3e_3,\\
  (j-1728)a_1^5 &= f_2^2e_2.
\end{align*}
If $u = \sqrt{-7}/7^4$, then the solution to the $j$-equations is:
\begin{align*}
  c_1 &= -952u, &c_2 &= 96u,\\
  c_3 &= -205797696u^2, &c_4 &=-5048u,\\
  c_5 &= -5904u, & c_6 &=426314304u^2,\\
  c_7 &= -2498515200000u^3, & c_8 &=7048u.
\end{align*}
The first few terms of the Fourier expansion of the hauptmodul are given in Table \ref{tab:hauptH3}.

\begin{table}
  \begin{tabular}{|c|c|}
    \hline
    $n$ & $a_n/u^{n+1}$\\
    \hline
    -1 &$ 1$ \\
    0 &$ 0$ \\
1&$    7583156$\\
2&$-8915200000$\\
3&$25855539541090$\\
4&$-38753899878400000$\\
5&$59853295754680171800$\\
6&$-107814623754600729600000$\\
7&$130691527974826975392903135$\\
8&$-229196454200112641389772800000$\\
9&$294346563065808045129145192319236$\\
10&$-427644716636763893188085418688000000$\\
11&$606586125578466006634487839969153168734$\\
    \hline
  \end{tabular}
  \caption{Normalized Fourier coefficients of the hauptmodul for $H_3$}\label{tab:hauptH3}
\end{table}

\subsection{The conjugates of $U_1$}
In this case the $j$-equations read
\begin{align*}
  ja_1 &= f_3^3e_3,\\
  (j-1728)a_1 &= f_2^2e_2.
\end{align*}
%
%
%
Let $\zeta_3$ denote a third root of unity and set
\[ u =  \left((1763\zeta_3 + 1255)2^2 3/7^7\right)^{(1/6)}. \]
Note that the minimal polynomial of $u$ over $\QQ$ is $823543X^{12}-8964X^6+432$. The solution to  the $j$-equations is then:
\begin{align*}
  c_1 &=(-8\zeta_3 - 10)u , &c_2 &= (-6\zeta_3 - 6)u,\\
  c_3 &= (-28\zeta_3 - 20)u^2, &c_4 &= (10\zeta_3 + 8)u,\\
  c_5 &= (-4\zeta_3 - 8)u, & c_6 &=(-60\zeta_3 + 12)u^2,\\
  c_7 &= (60\zeta_3 - 276)u^3, & c_8 &= 6u.
\end{align*}
The first few terms of the Fourier expansion of the hauptmodul are given in Table \ref{tab:hauptU1}. 
\begin{rmk}
In $\QQ(\sqrt{-3})$ the element $(1763\zeta_3 +1255)2^2 3/7^7$ has a non-trivial valuation at
$2$ of $2$, at $3$ of $3$ and at one of the two primes dividing $7$ of $-7$. The valuation is $0$ at all other primes of $\QQ(\sqrt{-3})$.
Consequently, the denominators of $u$ are at exactly one of the two primes over $7$.

The $2$'s and $3$'s appearing in the denominators are cancelled by those appearing in the numerator in the power of $u$.
\end{rmk}

\begin{table}
  \begin{tabular}{|c|c|}
    \hline
    $n$ & $a_n/u^{n+1}$\\
    \hline
    -1 &$ 1$ \\
    0 &$ 0$ \\
 1 & $20\zeta_3 + 4$\\
 2 & $60\zeta_3 + 12$\\
 3 & $-96\zeta_3 + 48$\\
 4 & $432\zeta_3 + 288$\\
 5 & $-(3893/9)\zeta_3 - 1060/9$\\
 6 & $576\zeta_3 - 576$\\
 7 & $(13952/3)\zeta_3 + 7372$\\
 8 & $7168\zeta_3 + 18312$\\
 9& $-45200\zeta_3 - 33568$\\
 10 & $-4160\zeta_3 + 93248$\\
 11 & $-(3412747/72)\zeta_3 - 22548985/216$\\
    \hline
  \end{tabular}
  \caption{Normalized Fourier coefficients of the hauptmodul for $U_1$}\label{tab:hauptU1}
\end{table}

\subsection{The conjugates of $U_6$}
In this case the $j$-equations read
\begin{align*}
  ja_1^6 &= f_3^3e_3,\\
  (j-1728)a_1^6 &= f_2^2e_2.
\end{align*}
As above let $\zeta_3$ denote a third root of unity and set
\[ u = \left((3\zeta_3+1)/7\right)^7. \]
The solution to the $j$-equations is:
\begin{align*}
  c_1 &=(-1368\zeta_3 - 4944)u , &c_2 &=(59472\zeta_3 + 238944)u,\\
  c_3 &=(738742464\zeta_3 + 1457337024)u^2, &c_4 &= (-1368\zeta_3 + 1968)u,\\
  c_5 &= (-128520\zeta_3 - 512496)u, & c_6 &=(-5453272512\zeta_3 - 13411016640)u^2,\\
  c_7 &=(-8345692154880\zeta_3 - 38174900673024)u^3, & c_8 &= (3816\zeta_3 + 5424)u.
\end{align*}
The first few terms of the Fourier expansion of the hauptmodul are given in Table \ref{tab:hauptU6}.
\begin{rmk}
The only prime of $\QQ(\sqrt{-3})$ at which $u$ has non-trivial valuation is one of the two primes over $7$.
It is has valuation $-7$ at this prime.
\end{rmk}

\begin{table}
  \begin{tabular}{|c|c|}
    \hline
    $n$ & $a_n/u^{n+1}$\\
    \hline
    -1 &$ 1$ \\
    0 &$ 0$ \\
1& $3195612\zeta_3 + 4653180$\\
2& $2113007616\zeta_3 + 7901431808$\\
3& $-5777884753902\zeta_3 - 11584189398816$\\
4& $3171254057975808\zeta_3 + 3027156411138048$\\
5& $-20391915647836108224\zeta_3 - 3800819906733485320$\\
6& $15478255418070783762432\zeta_3 - 10803276590128984571904$\\
7& $26591161128955478844327729\zeta_3 + 24908794926096718823786001$\\
8& $-26181911558676353382430801920\zeta_3 - 12727797977727574691751002112$\\
9& $26604087748477982557834447865556\zeta_3 - 15929436789742692451659751424160$\\
    \hline
  \end{tabular}
  \caption{Normalized Fourier coefficients of the hauptmodul for $U_6$}\label{tab:hauptU6}
\end{table}

\section{Some results on Eisenstein series}
\label{s:eisenstein}

One of our original aims was to see how much one could say about the Eisenstein series associated to these minimal noncongruence subgroups, but as pointed out by Philips-Sarnak \cite{PhilipsSarnak}, it is difficult if not impossible to say too much about their Fourier coefficients in general. See also \cite{Scholl1}, \cite{Scholl2} and \cite{MurtyRamakrishnan} where it is observed that even the \emph{algebraicity} of Fourier coefficients of Eisenstein series can be a thorny question. For example, by the main theorem of \cite{Scholl1}, the holomorphic Eisenstein series of weight $2$ associated to the noncongruence subgroup discussed in \cite{BKP} has infinitely many transcendental Fourier coefficients.

\begin{rmk}
  Note that by \cite{Scholl1}, the Eisenstein series of weight $2$ for the noncongruence subgroups of index $7$ have algebraic Fourier coefficients, since the Manin-Drinfeld condition is trivially satisfied in these cases: the Picard group of degree zero divisor classes on the compactified curve associated to each of these groups is cyclic of order $12$. More simply, algebraicity follows in weight $2$ because the space of forms of weight $2$ is one-dimensional in each case.
\end{rmk}

We begin by recalling a standard computation for the Fourier coefficients of Eisenstein series on any subgroup of $\Gamma$ of finite index; see \cite{Kubota} for more details. For simplicity in this section we focus solely on the group $G_1$. For even integers $k\geq 4$ define
\[
 g_k(\tau) \df g_{k}^{(\infty)}(\tau) = \sum_{\langle \pm T^4\rangle \backslash G_1} \frac{1}{(c\tau+d)^k},
\]
which converges absolutely for $k \geq 4$. Observe that elements of $\langle \pm T^4\rangle \backslash G_1$ are in one-to-one correspondence with the equivalence classes of elements in $G_1$ with the same bottom row (up to sign). Therefore we define
\[
  \chi(c,d) = \begin{cases}
    1 & \exists \stwomat **cd \in G_1,\\
    0 & \textrm{otherwise,} 
  \end{cases}
\]
and we find that
\[
  g_{k}^{(\infty)}(\tau) = 1+\sum_{c=1}^{\infty}\sum_{d=-\infty}^\infty \frac{\chi(c,d)}{(c\tau+d)^k}.
\]
\begin{rmk}
The above defines the Eisenstein series for the cusp at infinity, $g_k^{(\infty)}$. We can analogously define an Eisenstein series, $g_k^{(1)}$, for the cusp at one. It is elementary to verify that $g_k^{(\infty)} + g_k^{(1)} = E_k$. As such in what follows we will not typically consider $g_{k}^{(1)}$ and so we define $g_k = g_k^{(\infty)}$ for $k \ge 4$. For $k=2$ these series are not holomorphic modular forms and we instead consider $g_2 = g_2^{(\infty)}-\tfrac{4}{3}g_2^{(1)}$. 
\end{rmk}

\begin{prop}
  \label{p:indicator}
  The indicator function $\chi(c,d)$ satisfies the following properties:
  \begin{enumerate}
  \item $\chi(c,d) = \chi(c,d+4c)$ and $\chi(c,d) = \chi(c+4d,d)$;
  \item $\chi(c,d) = \chi(-c,-d)$;
  \item $\chi(c,d) = \chi(d,d-c)$;
  \item $\chi(c,d) = \chi(3c+2d,-5c-3d)$;
  \item $\chi(c,d) = \chi(-c,d-c)$;
  \item $\chi(c,d) = \chi(d,c)$.
  \end{enumerate}
\end{prop}
\begin{proof}
The first identity in Property (1) follows from the simple observation that
  \[
  \stwomat **cd T^4 = \stwomat**{c}{d+4c} 
\]
and $T^4 \in G_1$. The second follows likewise using $U = \stwomat 1011$, where $U^4 \in G_1$ is the minimal power in $G_1$. Properties (2), (3) and (4) are equivalent with $-1, R, E \in G_1$, respectively. For (5) we can use Lemma \ref{l:outer} and the fact that $T^{-1}G_2T = G_1$, so that $T^{-1}\psi T$ fixes $G_1$. Since $T^{-1}\psi T$ acts on bottom rows as $(c,d) \mapsto (-c,d-c)$, Property (5) follows. Now we can show that Property (6) is a consequence of the other properties:
\[
  \chi(c,d) = \chi(-c,d-c) = \chi(c,c-d) = \chi(d,c). \qedhere
\]
\end{proof}

\begin{rmk}
We make no use of Property (4) stated above, but we include it for completeness, as Properties (1) through (4) in Proposition \ref{p:indicator} encode the action of the generators of $G_1$. Properties (5) and (6) are somewhat less trivial, as they utilize the symmetry of the outer automorphism discussed above.
\end{rmk}

Given Proposition \ref{p:indicator}, we can simplify the expression for $g_k^{(\infty)}$:
\begin{align*}
  g_{k}^{(\infty)}(\tau)& = 1+\sum_{c=1}^{\infty}\sum_{d=-\infty}^\infty \frac{\chi(c,d)}{(c\tau+d)^k}\\
           &=1+\sum_{c=1}^{\infty}\sum_{d=1}^{4c}\sum_{t=-\infty}^\infty \frac{\chi(c,d+4ct)}{(c\tau+d+4ct)^k}\\
           &=1+\sum_{c=1}^{\infty}\frac{1}{(4c)^k}\sum_{d=1}^{4c}\chi(c,d)\sum_{t=-\infty}^\infty \frac{1}{(\frac{c\tau+d}{4c}+t)^k}\\
           &=1+\frac{(2\pi i)^k}{(k-1)!}\sum_{c=1}^{\infty}\frac{1}{(4c)^k}\sum_{d=1}^{4c}\chi(c,d)\sum_{n=1}^\infty n^{k-1}e^{2\pi i n(c\tau+d)/4c}\\
  &= 1+\frac{(2\pi i)^k}{(k-1)!}\sum_{n=1}^{\infty}\sum_{c=1}^{\infty}n^{k-1}\frac{1}{(4c)^k}\left(\sum_{d=1}^{4c}\chi(c,d)e^{2\pi i nd/4c}\right)e^{2\pi i n\tau/4}
\end{align*}
Thus we obtain the Fourier expansion:
\begin{equation}
  \label{eq:fourier}
g_{k}^{(\infty)}(\tau) = 1+\frac{(2\pi i)^k}{4^k(k-1)!}\sum_{n=1}^{\infty}n^{k-1}\left(\sum_{c=1}^{\infty}\left(\sum_{d=1}^{4c}\chi(c,d)e^{2\pi i nd/4c}\right)\frac{1}{c^k}\right)q_4^n.
\end{equation}
In particular, if we define
\begin{align*}
  X(n,c) &\df\sum_{d=1}^{4c}\chi(c,d)e^{2\pi i nd/4c},\\
  D(n,s) &\df \sum_{c=1}^{\infty}\frac{X(n,c)}{c^s},
\end{align*}
then the Fourier coefficients are:
\begin{equation*}
  a_n = \left(\frac{n\pi i}{2}\right)^k\frac{D(n,k)}{n(k-1)!} .
\end{equation*}
Thus, computation of Fourier coefficients is reduced to the evaluation of special values of the Dirichlet series $D(n,s)$. To aid in evaluating such series numerically we provide an algorithm for computing $\chi(c,d)$:
\begin{enumerate}
\item lift $(c,d)$ to a matrix $g=\stwomat abcd \in \Gamma$ using the extended Euclidean algorithm;
\item test if $g$, $Tg$ , $T^2g$ or $T^3g$ is in $G_1$; if so, $\chi(c,d) = 1$ and if not then $\chi(c,d) = 0$.
\item To test if a matrix $h \in \Gamma$ lies in $G_1$, write it as a word in $S$ and $T$, then obtain the analogous word in $\phi_1(S),\phi_1(T) \in S_7$, and check whether the resulting permutation fixes $1$. If so $h \in G_1$, and not otherwise.
\end{enumerate}
See Appendix \ref{a:code} for a \emph{Pari} code implementation.



\begin{rmk}\label{rem:convergence}
To obtain a very rough estimate for $D(n,k)$ computed using the values $c\leq N$, call this approximation $S_N$, observe that
\[
  \abs{D(n,k)-S_N} \leq \sum_{c>N} \frac{4}{c^{k-1}}< 4\int_{N}^\infty \frac{dx}{x^{k-1}} =\frac{4}{(k-2)N^{k-2}}.
\]
In particular, the number of digits of accuracy in this approximation is at least $k-2$ times the number of digits in $N$.
 In practice it appears that $X(n,c) <\!\!< c$, see Figure \ref{f:dirichlet}. This results that is slightly better than this. 
This also leads to the apparent absolute convergence of the series for $k=2$. 
None the less the naive method for evaluating the Fourier coefficients is in general quite inefficient for small $k$.
\end{rmk}

\begin{rmk}
  \label{r:4ctoc}
Observe that 
\begin{align*}
  X(n,c)=& \sum_{d=1}^c\left(\chi(c,d)+i^n\chi(c,d+c)+(-1)^n\chi(c,d+2c)+(-i)^n\chi(c,d+3c)\right)e^{\tfrac{2\pi i nd}{4c}}.
\end{align*}
Evaluating $\chi(c,d)$ involves solving a word problem, and the solution to that word problem can be used to solve the corresponding word problems involved in evaluating $\chi(c,d+c)$, $\chi(c,d+2c)$ and $\chi(c,d+3c)$. In this way one can speed up the evaluation of the approximation to $D(n,k)$ by precomputing values $\chi(c,d)$ four at a time using this optimization.
\end{rmk}

\begin{rmk}
For each group being considered there is a unique (normalized) modular form in weight $2$. A simple computation with the divisors reveals that in every case it will be given precisely by
$g_2 = (E_6\cdot e_3 \cdot f_3)/(E_4\cdot f_2)$.
In Table \ref{tab:fourier} we give the first few Fourier coefficients as a series in respectively $q_4$, $q_3$, $q_5$, and $q_2$.
\begin{table}
\begin{center}
 \resizebox{.88\textwidth}{!}{ \begin{tabular}{|c|c||c|}
  \hline
  $n$ & $G_1$:\quad $a_n/u^n$, $u=\sqrt[4]{-7}/7^2$ & $G_3$:\quad $a_n/u^n$, $u=\sqrt[3]{-2/7}/7^2$ \\
  \hline
$0$&$1$&$1$\\
$1$&$-168$&$462$\\
$2$&$-840$&$-84420$\\
$3$&$733152$&$-807828$\\
$4$&$-1615656$&$-891458736$\\
$5$&$1179184272$&$82305718992$\\
$6$&$-5780133408$&$5155138704870$\\
$7$&$-1097701319232$&$807981764899218$\\
$8$&$20620554819480$&$-57396539567144736$\\
$9$&$-1310614136578824$&$829520378016134700$\\
$10$&$-14959868841286320$&$-368800915551641445600$\\
    \hline
    \hline
 $n$ & $H_1$:\quad $a_n/u^n$, $u=\sqrt[5]{-7^3}/7^2$ & $H_3$:\quad $a_n/u^n$, $u=\sqrt[2]{-7}/7^4$ \\
  \hline
$0$&$1$&$1$\\
$1$&$-28$&$952$\\
$2$&$-3108$&$-14260008$\\
$3$&$88172$&$5950907872$\\
$4$&$824012$&$18866241755032$\\
$5$&$-14260008$&$14858201843068752$\\
$6$&$352362948$&$-29392973490650091168$\\
$7$&$13569079384$&$18769317912571342452672$\\
$8$&$-195382795860$&$26663537479505346618394392$\\
$9$&$-1200557668744$&$12713310504973377181575454552$\\
$10$&$18866241755032$&$-36194240778558471635244990599408$\\
    \hline
  \end{tabular}}
%
\resizebox{.89\textwidth}{!}{
  \begin{tabular}{|c|c|}
      \hline
 $n$ & $U_1$:\quad $a_n/u^n$, $u=\left((1763\zeta_3 + 1255)2^2 3/7^7\right)^{(1/6)}$ \\
 \hline
0&$ 1$\\
1&$ 8\zeta_3 + 10$\\
2&$ 56\zeta_3 + 28$\\
3&$ 84\zeta_3 - 84$\\
4&$ -336$\\
5&$ -1008\zeta_3 - 1008$\\
6&$-710/3\zeta_3 - 184/3$\\
7&$-9566/3\zeta_3 + 9088/3$\\
8&$4256\zeta_3 + 30016/3$\\
9&$ 15624\zeta_3 + 19404$\\
10&$139552/3\zeta_3 - 25984/3$\\
\hline
\hline
 $n$ & $U_6$:\quad $a_n/u^n$, $u= \left((3\zeta_3+1)/7\right)^7$ \\
 \hline
0&$1$\\
1&$1368\zeta_3 + 4944$\\
2 &$5264136\zeta_3 + 13265352$\\
3&$12839470272\zeta_3 + 16044542112$\\
4&$22545390152664\zeta_3 + 27018559576704$\\
5&$9748947084182352\zeta_3 + 9649676839772016$\\
6&$34718972026438197504\zeta_3 + 16480296599809346784$\\
7&$9778372812649484494272\zeta_3 + 9122178274543742453376$\\
8 &$35207674866620513785843560\zeta_3 + 3599167618394097606994536$\\
9 &$35212791025867821428233261296\zeta_3 - 1534671671263749769838754840$\\
10 &$19858438209488318852697458205264\zeta_3 - 8424036363723923387197847067264$\\
\hline
\end{tabular}
}
\end{center}
  \caption{Normalized Fourier coefficients of the forms $g_2$.}\label{tab:fourier}
  
  \end{table}

Assuming for a moment the convergence of $D(n,2)$ then with $g_2^{(\infty)}$ the function formally defined by Equation \eqref{eq:fourier} and $g_2^{(1)}$ the analogous function defined for the cusp $1$ one has that $g_2^{(\infty)}-\tfrac{4}{3}g_2^{(1)}$
will define a holomorphic modular form for $G_1$, and consequently $g_2 = g_2^{(\infty)}-\tfrac{4}{3}g_2^{(1)}$. 
Noting that $g_2^{(\infty)} + g_2^{(1)} = E_2$ this would allow us to immediately deduce that the special value $D(n,2)$ is in fact algebraic. This of course agrees with the expectations from \cite{Scholl1} and is analogous to the Proposition on page 260 of \cite{MurtyRamakrishnan}.
\end{rmk}

Although it is generally nontrivial to determine the field of definition of Eisenstein series, or even if they are algebraic, in this case we can at least determine the phase of the Fourier coefficients:
\begin{prop}
For all $c \geq 1$ we have $X(n,c) \in (\RR\cap \bar \QQ)\cdot \zeta_8^n$, and so $a_n \in \RR\cdot \zeta_8^n$.
\end{prop}
\begin{proof}
Take a complex conjugate of the identity in Remark \ref{r:4ctoc} and use the properties in Proposition \ref{p:indicator} to obtain:
  \begin{align*}
    &\overline{\sum_{d=1}^c\left(\chi(c,d)+i^n\chi(c,d+c)+(-1)^n\chi(c,d+2c)+(-i)^n\chi(c,d+3c)\right)e^{\tfrac{2\pi i nd}{4c}}}\\
    =&\sum_{d=1}^c\left(\chi(c,d)+(-i)^n\chi(c,d+c)+(-1)^n\chi(c,d+2c)+i^n\chi(c,d+3c)\right)e^{\tfrac{2\pi i n(3c+c-d)}{4c}}\\
    =&\sum_{d=1}^c\left(\chi(c,c-d)+(-i)^n\chi(c,2c-d)+(-1)^n\chi(c,3c-d)+i^n\chi(c,4c-d)\right)e^{\tfrac{2\pi i n(3c+d)}{4c}}\\
    =&\sum_{d=1}^c\left(\chi(-c,d+3c)+(-i)^n\chi(-c,d+2c)+(-1)^n\chi(-c,d+c)+i^n\chi(-c,d)\right)(-i)^ne^{\tfrac{2\pi i nd}{4c}}\\
    =&\sum_{d=1}^c\left(\chi(c,d)+(-i)^n\chi(c,d+3c)+(-1)^n\chi(c,d+2c)+i^n\chi(c,d+c)\right)(-i)^ne^{\tfrac{2\pi i nd}{4c}}
  \end{align*}
 which shows that $\overline{X(n,c)} = (-i)^nX(n,c)$. Thus,
  \[
  \overline{X(n,c)(1+i)^{-n}} = (-i)^nX(n,c)(1-i)^{-n} = X(n,c)(1+i)^{-n}.
\]
Hence $X(n,c) \in (\RR \cap \overline{\QQ})\cdot \zeta_8^n$ as claimed. The second claim follows immediately from this.
\end{proof}
The Dirichlet series $D(n,s)$ are quite mysterious, as their coefficients $X(n,s)$ lie in increasingly large number fields, as opposed to more typical Dirichlet $L$-series, or Dedekind $\zeta$-functions, and so many standard techniques cannot be brought to bear on $D(n,s)$. 
It appears that perhaps $X(n,c) = O(c^{5/7})$, and Figure \ref{f:dirichlet} on page \pageref{f:dirichlet} shows a plot of some values that supports this. 
More precisely, when $n=1$, we have $\abs{X(1,c)} < c^{5/7}$ for all $32,769 < c<2,000,000$ (this bound fails for $15$ values below $32,769$). 
Likewise for $n=2,\ldots,11$ with $c<300,000$ and  $n=12,\ldots,50$ with $c<100,000$ the only values with $X(n,c)>c^{5/7}$ come from small values $c$.
Experimentally, see again Figure \ref{f:dirichlet} but also Figures \ref{f:dirichlet_coef_2} and \ref{f:dirichlet_coef_3}, it is evident that the distribution of the values $X(n,c)$ along the line $u^n$ is broadly controlled by the congruence $c \pmod{12}$. 
The exact distributions appear to depend on $n$: the cases for $X(1,c)$ are illustrated in Figures \ref{f:dirichlet_coef_2} and \ref{f:dirichlet_coef_3}.
We note that the exponent $2/7$ on $c$ is selected to make the graphs appear approximately normal, we have no evidence this is the correct exponent, nor that these distributions should be normal.
\begin{figure}

  \includegraphics[scale=0.14]{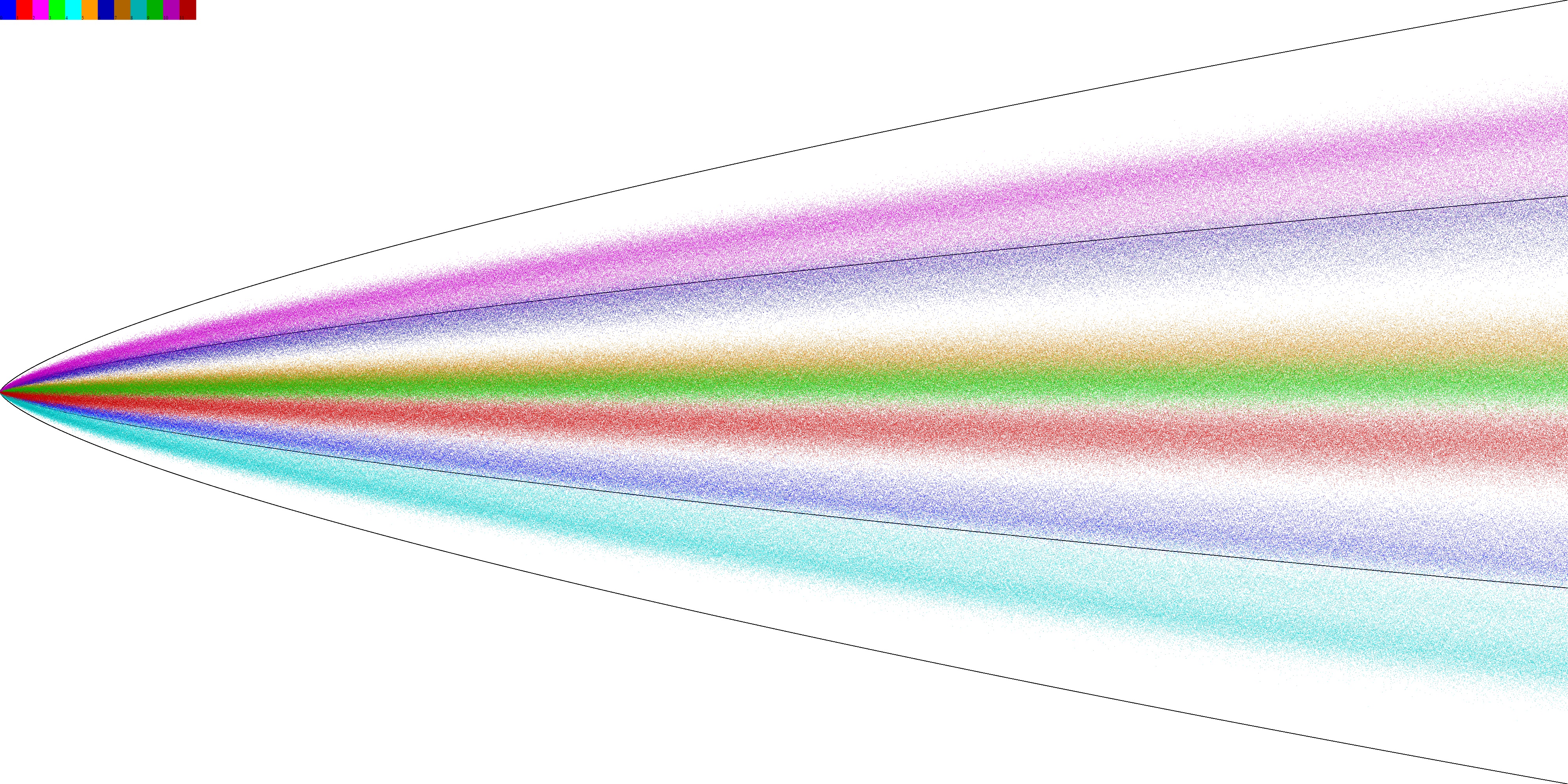}
  \caption{$X(1,c)/e^{(\pi i/4)}$ for $c\leq 2,000,000$. The outer black curve is $\pm c^{5/7}$, the inner is $\pm (1/2)c^{5/7}$.
                  From top to bottom the colored bands are $\pm2 \pmod{12}$, $6 \pmod{12}$, $\pm 5 \pmod{12}$, $\pm 3\pmod{12}$, $\pm1\pmod{12}$, $0\pmod{12}$, and $\pm4\pmod{12}$.}  \label{f:dirichlet}
                
\end{figure}
\begin{figure}
  \includegraphics[scale=0.475]{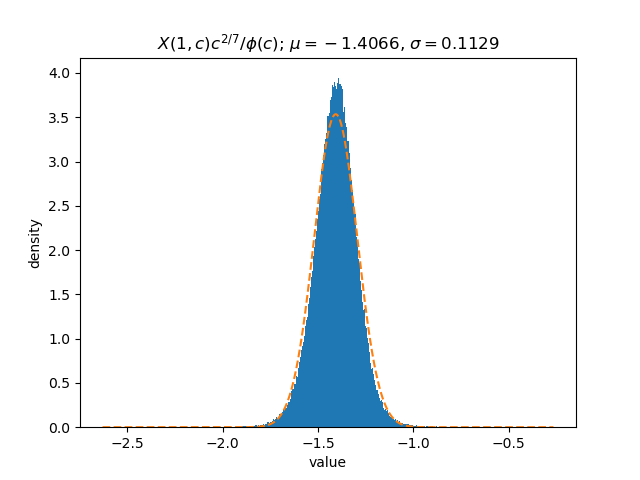}
 \includegraphics[scale=0.475]{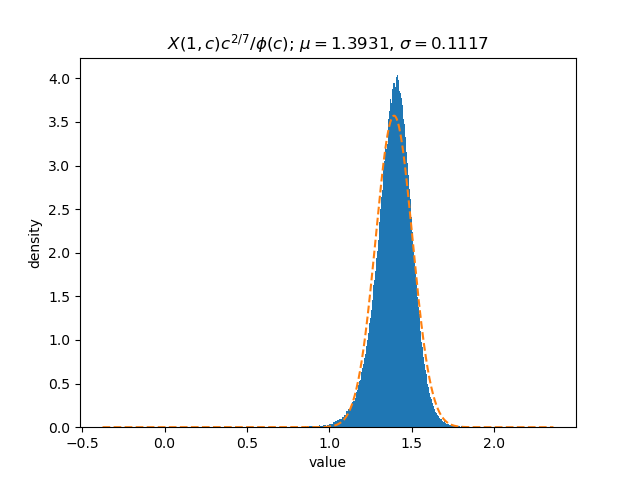}
   \caption{Normalized histograms for $X(1,c)c^{2/7}/(\phi(c) e^{\pi i/4})$ for $2||c$, $4|c$ with $c$ up to 2,000,000. Curve is a normal distribution with given paramters.}\label{f:dirichlet_coef_2}
 
 \includegraphics[scale=0.475]{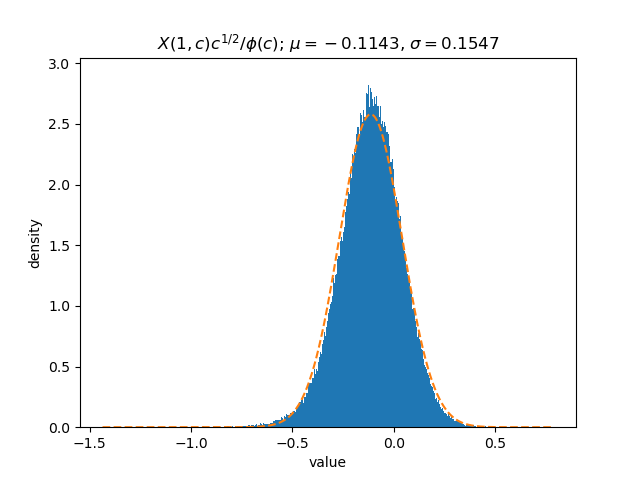}
 \includegraphics[scale=0.475]{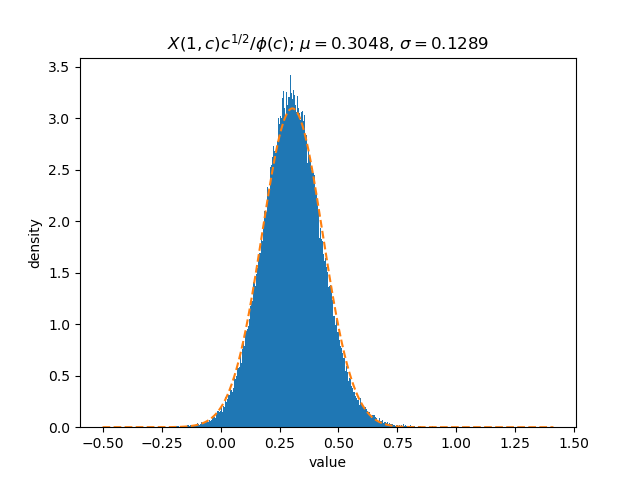}
 \includegraphics[scale=0.475]{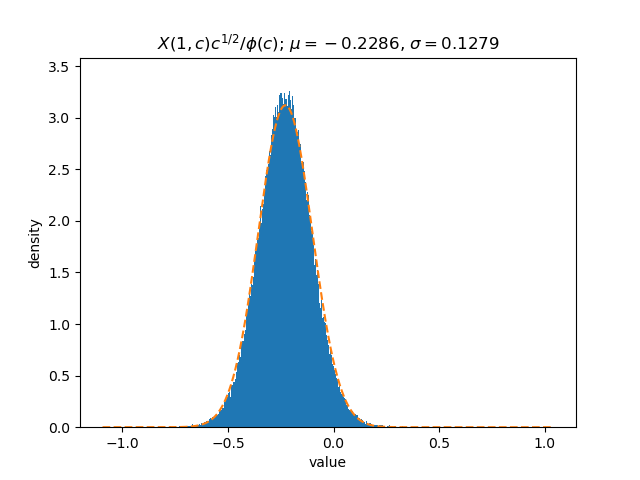}
   \caption{Normalized histograms for $X(1,c)c^{1/2}/(\phi(c) e^{\pi i/4})$ for (clockwise from top left) $c=\pm 3 \pmod{12}$, $c=\pm1\pmod{12}$, and $c=\pm 5\pmod{12}$  with $c$ up to 2,000,000.  Curve is a normal distribution with given paramters.} \label{f:dirichlet_coef_3}

\end{figure}

As with $g_2$ it is an exercise to conclude that $g_4 = (E_4/f_3) \cdot a_1 \cdot (z-Cu)$ for some $C\in \mathbb{R}$. The algebraicity of $C$ is equivalent to that of both the divisor of $g_4$ as well as that of its Fourier coefficients.
We have included for the curious reader our computations of the first few Fourier coefficients of $g_4$ in Table \ref{t:fourierg4} on page \pageref{t:fourierg4}. The computations for $a_1$ used terms with $c$ up to $2,000,000$, while for $a_2,\ldots,a_{10}$, we used $c$ up to $300,000$. The computations for $a_1$ took over a month of CPU time using resources from Compute Canada. Note that the final digits may not be accurate, as we have provided one digit beyond the apparent precision, and the actual precision may be less still (see Remark \ref{rem:convergence}). We have not been able to identify any apparent algebraic dependency for the higher coefficients, but this may be a simple reflection of a lack of sufficient precision to detect dependency relations.
\begin{table}
    \begin{tabular}{|c|rl|}
    \hline
  $n$&$a_n/u^n$,&\quad $u=(-7)^{(1/4)}/7^2$ \\
    \hline
    $0$&$1$ &\\
    $1$&$40$&\hskip-8pt$.7303189636318364926\cdots$\\
    $2$&$303$&\hskip-8pt$.7319312003984\cdots$\\
    $3$&$-1113445$&\hskip-8pt$.924994532325\cdots$\\
    $4$&$-101378021$&\hskip-8pt$.6026120116\cdots$\\
    $5$&$-4677356098$&\hskip-8pt$.49752275\cdots$\\
    $6$&$110516113983$&\hskip-8pt$.5601513\cdots$\\
    $7$&$10622672944963$&\hskip-8pt$.34244\cdots$\\
    $8$&$703827515349172$&\hskip-8pt$.972\cdots$\\
    $9$&$20587451911329502$&\hskip-8pt$.7\cdots$\\
    $10$&$54985771355001805$&\hskip-8pt$.6\cdots$\\
    \hline
    \end{tabular}
    \caption{Approximate values of normalized Fourier coefficients for $g_{4}$ for the group $G_1$}
   \label{t:fourierg4}
 \end{table}

  \appendix
  \section{Code}
  \label{a:code}

The following PARI/GP code computes $\chi(c,d)$
  
\begin{verbatim}
Chi(c,d) = {
   my(s=Vecsmall([2,1,4,3,6,5,7])); /* Permutation Phi(S) */
   my(t=Vecsmall([2,4,6,5,1,7,3])); /* Permutation Phi(T) */
   my(res=Vecsmall([1,2,3,4,5,6,7]));
   my(q);
   while( c != 0,  /* Compute Phi(M) for M with bottom row c,d */
       if( abs(d) >= abs(c) ,
            q=d\c;
            d=d-q*c;
            res = (t^(q%12))*res;
         ,
            res = s*res;
            q=c;
            c=d;
            d=-q;
       );
   );
   if( abs(d) != 1, return(0)); /* c,d  not relatively prime */
   if( res[1]==1||res[1]==t[1]||res[1]==(t*t)[1]||res[1]==(t*t*t)[i],
           return(1);  
   );
   return(0);
}
\end{verbatim}
  
\bibliographystyle{plain}
\bibliography{refs}
\end{document}